\documentclass[a4paper,11pt]{amsart}

\usepackage[french, ngerman, italian, english]{babel}
\usepackage{amsmath,amsthm, amssymb}
\usepackage{setspace}
\usepackage{anysize}
\usepackage{url}
\usepackage{bm}
\usepackage{graphicx}
\usepackage[colorlinks=true]{hyperref}
\usepackage{mathptmx}
\usepackage{paralist}
\usepackage{verbatim}
\usepackage[utf8]{inputenc}

\theoremstyle{plain}
\newtheorem{thm}{\bf Theorem}[section]

\newtheorem{prop}[thm]{\bf Proposition}
\newtheorem{lemma}[thm]{\bf Lemma}
\newtheorem{corollary}[thm]{\bf Corollary}

\theoremstyle{definition}
\newtheorem{definition}[thm]{\bf Definition}

\theoremstyle{remark}
\newtheorem{remark}[thm]{\bf Remark}
\newtheorem{example}[thm]{\bf Example}
\theoremstyle{example}

\def\NZQ{\Bbb}               
\def\NN{{\NZQ N}}

\def\ZZ{{\NZQ Z}}
\def\RR{{\NZQ R}}
\def\CC{{\NZQ C}}

\def\frk{\frak}               

\def\pp{{\frk p}}

\def\mm{{\frk m}}

\def\Phi{{\frk n}}
\def\Phi{{\frk N}}
\def\Jc{{\mathcal J}}

\def\Rc{{\mathcal R}}
\def\Cc{{\mathcal C}}
\def\ab{{\bold a}}

\def\opn#1#2{\def#1{\operatorname{#2}}} 
\opn\chara{char} \opn\length{\ell} \opn\pd{pd} \opn\rk{rk}
\opn\projdim{proj\,dim} \opn\injdim{inj\,dim} \opn\rank{rank}
\opn\depth{depth} \opn\grade{grade} \opn\height{height}
\opn\embdim{emb\,dim} \opn\codim{codim}
\def\OO{{\mathcal O}}
\opn\Tr{Tr} \opn\bigrank{big\,rank}
\opn\superheight{superheight}\opn\lcm{lcm}
\opn\trdeg{tr\,deg}
\opn\reg{reg} \opn\lreg{lreg} \opn\ini{in} \opn\lpd{lpd}
\opn\size{size} \opn\sdepth{sdepth}
\opn\link{link}\opn\fdepth{fdepth}\opn\lex{lex}
\opn\div{div} \opn\Div{Div} \opn\cl{cl} \opn\Cl{Cl}
\opn\Spec{Spec} \opn\Supp{Supp} \opn\supp{supp} \opn\Sing{Sing}
\opn\Ass{Ass} \opn\Min{Min}\opn\Mon{Mon}
\opn\Ann{Ann} \opn\Rad{Rad} \opn\Soc{Soc}
\opn\Im{Im} \opn\Ker{Ker} \opn\Coker{Coker} \opn\Am{Am}
\opn\Hom{Hom} \opn\Tor{Tor} \opn\Ext{Ext} \opn\End{End}
\opn\Aut{Aut} \opn\id{id}

\opn\nat{nat}
\opn\pff{pf}
\opn\Pf{Pf} \opn\GL{GL} \opn\SL{SL} \opn\mod{mod} \opn\ord{ord}
\opn\Gin{Gin} \opn\Hilb{Hilb}\opn\sort{sort}
\opn\aff{aff} \opn\con{conv} \opn\relint{relint} \opn\st{st}
\opn\lk{lk} \opn\cn{cn} \opn\core{core} \opn\vol{vol}
\opn\link{link} \opn\star{star}\opn\lex{lex}\opn\set{set}
\opn\gr{gr}

\def\pot#1#2{#1[\kern-0.28ex[#2]\kern-0.28ex]}

\opn\dirlim{\underrightarrow{\lim}}
\opn\inivlim{\underleftarrow{\lim}}
\def \GL{{\operatorname{GL}}}
\def \Sym{{\operatorname{Sym}}}

\def \chara{{\operatorname{char}}}

\def \height{{\operatorname{ht}}}
\def \reg{{\operatorname{reg}}}
\def \depth{{\operatorname{depth}}}
\def \Gin{{\operatorname{Gin}}}
\def \grade{{\operatorname{grade}}}
\def \Ker{{\operatorname{Ker}}}
\def \pd{{\operatorname{pd}}}
\def \mm{{\mathfrak{m}}}

\def \NN{\mathbb N}

\def \ZZ{\mathbb Z}

\def \B{\mathcal B}

\def \S{\mathcal S}

\def \chara{{\operatorname{char}}}
\def \init{{\operatorname{in}}}

\def \height{{\operatorname{ht}}}
\def \reg{{\operatorname{reg}}}
\def \depth{{\operatorname{depth}}}
\def \Gin{{\operatorname{Gin}}}
\def \grade{{\operatorname{grade}}}
\def \Ker{{\operatorname{Ker}}}
\def \pd{{\operatorname{pd}}}
\def \mm{{\mathfrak{m}}}

\def \NN{\mathbb N}
\def \ZZ{\mathbb Z}

\def \S_d{\mathcal{M}(d)}

\def \S{\mathcal S}
\def \P{\mathcal P}
\def \H{\mathcal H}

\def \b{\bullet}

\def \StAss{\overline{\mathrm{StAss}}}
\def \fpt{\mathrm{fpt}}
\def \lct{\mathrm{lct}}
\def \gin{\mathrm{gin}}

\def \tl{{}^{\textup{t}\negthinspace}}

\def\ini{\operatorname{\rm in}}

\begin{document}
\title{Test, multiplier and invariant ideals}
\author{In\^es Bonacho Dos Anjos Henriques}
\thanks{The authors were supported by the EPSRC grant EP/J005436/1 (IBH), and  by PRIN  2010S47ARA\_003 ``Geometria delle Variet\`a Algebriche" (MV).}
\address{School of Mathematics and Statistics, 
University of Sheffield, 
Hounsfield Road, Sheffield S3 7RH, United Kingdom.}
\email{i.henriques@sheffield.ac.uk}
\author{Matteo Varbaro}
\address{Dipartimento di Matematica, 
Universit\`a degli Studi di Genova, 
Via Dodecaneso 35, 16146 Genova, Italy}
\email{varbaro@dima.unige.it}
\date{}
\maketitle

\begin{abstract}
This paper gives an explicit formula for the multiplier ideals, and consequently for the log canonical thresholds, of any $\GL(V)\times \GL(W)$-invariant ideal in $S=\Sym(V\otimes W^*)$, where $V$ and $W$ are vector spaces over a field of characteristic 0. This characterization is done in terms of a polytope constructed from the set of Young diagrams corresponding to the Schur modules generating the ideal. 

Our approach consists in computing the test ideals of some invariant ideals of $S$ in positive characteristic: Namely, we will compute the test ideals (and so the $F$-pure thresholds) of any sum of products of determinantal ideals. Even in characteristic 0, not all the invariant ideals are as the latter, but they are up to integral closure, and this is enough to reach our goals.

The results concerning the test ideals are obtained as a consequence of general results holding true in a special situation. Within such framework fall determinantal objects of a generic matrix, as well as of a symmetric matrix and of a skew-symmetric one. Similar results are thus deduced for the $\GL(V)$-invariant ideals in $\Sym(\Sym^2V)$ and in $\Sym(\bigwedge^2V)$. (Also monomial  ideals fall in this framework, thus we recover Howald's formula for their multiplier ideals and, more generally, we get the formula for their test ideals). During the proof, we introduce the notion of ``floating test ideals", a property that in a sense is satisfied by ideals defining schemes with singularities as nice as possible. As we will see, products of determinantal ideals, and by passing to characteristic 0 ideals generated by a single Schur module, have this property. 
\end{abstract}

\section{Introduction}

Given an ideal $I\subseteq K[x_1,\ldots ,x_N]$, where $K$ is a field of characteristic $0$, its {\it multiplier ideals} $\Jc(\lambda\b I)$ (where $\lambda\in\RR_{>0}$) are defined by meaning of a log-resolution. The {\it log-canonical threshold} of $I$ is just the least $\lambda$ such that $\Jc(\lambda\b I)\subsetneq K[x_1,\ldots ,x_N]$. In the words of Lazarsfeld \cite{La2}, ``the intuition is that these ideals will measure the singularities of functions $f\in I$, with `nastier' singularities being reflected in `deeper' multiplier ideals''. In this paper, we will give explicit formulas for the multiplier ideals (and therefore for the log-canonical thresholds) of all the $G$-invariant ideals in the following polynomial rings $S$ over a field of characteristic 0:

\begin{itemize}
\item[(i)] $S=\Sym(V\otimes W^*)$, where $V$ and $W$ are finite $K$-vector spaces, $G=\GL(V)\times \GL(W)$ and the action extends the diagonal one on $V\otimes W^*$ (Theorem \ref{thm:mult}).
\item[(ii)] $S=\Sym(\Sym^2V)$, where $V$ is a finite $K$-vector spaces, $G=\GL(V)$ and the action extends the natural one on $\Sym^2V$ (Theorem \ref{thm:multsym}). 
\item[(iii)] $S=\Sym(\bigwedge^2V)$, where $V$ is a finite $K$-vector spaces, $G=\GL(V)$ and the action extends the natural one on $\bigwedge^2V$ (Theorem \ref{thm:multpf}). 
\end{itemize} 

\vspace{2mm}

The above results are obtained via reduction to characteristic $p>0$: If $I\subseteq K[x_1,\ldots ,x_N]$, where $K$ is a field of characteristic $p$, its {\it (generalized) test ideals} $\tau(\lambda\b I)$ (where $\lambda\in\RR_{>0}$) are defined by using tight closure ideas involving the Frobenius endomorphism. The connection between multiplier and test ideals is given by Hara and Yoshida \cite{Hara-Yoshida}, in a sense explaining why statements originally proved by using the theory of multiplier ideals often admit a proof also via the Hochster-Huneke theory of tight closure \cite{HH:JAMS}: Roughly speaking, if $p\gg 0$, test ideals and (the reduction mod $p$ of) multiplier ideals are the ``same''. We give a general result for computing all test ideals of classes of ideals $I$ satisfying certain conditions in polynomial rings $S$ over a field of characteristic $p>0$ (Theorem \ref{thm:main2}). To give an idea, such
conditions, quite combinatorial in nature, involve the existence of a polytope controlling the integral closure of the powers of $I$, and the existence of a pair consisting in a polynomial of $S$ and in a term ordering on $S$ having properties depending on the coordinates of the real vector space in which the polytope lives (which correspond to suitable $\pp\in\Spec(S)$) and their weights (which are $\height(\pp)$) (see \ref{def:ast} for the precise definition). One can show that these conditions are satisfied by the following classes of ideals, for whose test ideals we therefore obtain explicit formulas (and so for the $F$-pure thresholds, that are interestingly independent on the characteristic of the base field):

\begin{itemize}
\item[(i)] Ideals $I\subseteq S=K[X]$, where $X$ is a generic matrix, which are sums of products of determinantal ideals of $X$ (Corollary \ref{cor:genericsum}).
\item[(ii)] Ideals $I\subseteq S=K[Y]$, where $Y$ is a symmetric matrix, which are sums of products of determinantal ideals of $Y$ (Corollary \ref{cor:symmetricsum}).
\item[(iii)] Ideals $I\subseteq S=K[Z]$, where $Z$ is a skew-symmetric matrix, which are sums of products of Pfaffian ideals of $Z$ (Corollary \ref{cor:skew-symmetricsum}). 
\end{itemize} 

The polynomial rings with the $G$-actions described at the beginning, of course, can be defined in any characteristic. Indeed, there are $G$-equivariant isomorphisms with the above polynomial rings endowed with suitable actions. With respect to such suitable actions, the above ideals $I$ are $G$-invariant, although there are many more $G$-invariant ideals (even in characteristic 0); on the other hand they are ``enough", essentially thanks to results obtained by DeConcini, Eisenbud and Procesi in \cite{DEP} (also the classification of the $G$-invariant ideals of $\Sym(V\otimes W^*)$, in characteristic 0, is in \cite{DEP}). The described results broadly generalize theorems of: 
\begin{itemize}
\item[(i)] Johnson \cite{Johnson} who in her PhD thesis computed the multiplier ideals of determinantal ideals, which are evidently $G$-invariant ideals of $\Sym(V\otimes W^*)$. 
\item[(ii)] Docampo \cite{Docampo}, who computed the log-canonical threshold of determinantal ideals using different methods to the one used by Johnson.
\item[(iii)] Miller, Singh and Varbaro \cite{MSV}, who computed the $F$-pure threshold of determinantal ideals.
\item[(iv)] Henriques \cite{He}, who computed the test ideals of the determinantal ideal generated by the maximal minors of  the matrix $X$. 
\end{itemize}
Theorem \ref{thm:main2} does not concern only determinantal objects: also monomial ideals satisfy the condition of Definition \ref{def:ast}, being that the integral closure of monomial ideals is controlled by the Newton polytope. As an immediate consequence, we obtain a formula for the test ideals of a monomial ideal (Remark \ref{rem:mon}). In particular, we recover the formula for the multiplier ideals of a monomial ideal established by Howald in \cite{Howald}.

Of course, from the results described above, one can read all the jumping numbers for the multiplier ideals, as well as the $F$-jumping numbers, of all the involved ideals. Interestingly, these invariants agree independently of the characteristic.

\vspace{2mm}

The results described above are included in Section \ref{sec:mult} (the last section). In Section \ref{sec:floating}, we prove that the test ideals $\tau(\lambda\b I)$ are always contained in an ideal defined through a valuation, depending on $I$, on $\Spec(S)$ (Proposition \ref{prop:generalinclusion}).
This motivates the introduction of the class of ideals with {\it floating test ideals} as the ideals for which the equality in Proposition \ref{prop:generalinclusion} holds (Definition \ref{def:floating}). In a sense we can say that ideals with floating test ideals define schemes with  singularities as nice as possible. Also, in this case, we can identify a class of ideals of $S$ having floating test ideals (Theorem \ref{thm:main}). As a corollary, we get that the following classes of ideals have floating test ideals:

\begin{itemize}
\item[(i)] Ideals $I\subseteq S=K[X]$, where $X$ is a generic matrix, which are products of determinantal ideals of $X$ (Corollary \ref{cor:generic}).
\item[(ii)] Ideals $I\subseteq S=K[Y]$, where $Y$ is a symmetric matrix, which are products of determinantal ideals of $Y$ (Corollary \ref{cor:symmetric}).
\item[(iii)] Ideals $I\subseteq S=K[Z]$, where $Z$ is a skew-symmetric matrix, which are products of Pfaffian ideals of $Z$ (Corollary \ref{cor:skew-symmetric}). 
\end{itemize} 
In characteristic 0, by defining the class of ideals having {\it floating multiplier ideals} in an analogous way, we have that the ideals of $\Sym(V\otimes W)$, $\Sym(\Sym^2V)$ and $\Sym(\bigwedge^2V)$ generated by an irreducible $G$-representation have floating multiplier ideals.

\vspace{2mm}

In Section \ref{sec:basics}, we will recall the tools needed from representation theory and ASL (Algebras with Straightening Law) theory, the definition of multiplier and test ideals, and some basic properties of test ideals.

\bigskip

\noindent {\bf Acknowledgements}: The second author of this paper would like to thank Mihnea Popa for bringing his attention, after having attended a seminar on this work, to Howald's article \cite{Howald}, leading to the writing of Remark \ref{rem:mon}.

\section{Setting the table}\label{sec:basics}

Throughout, $N$ will be a positive integer, $K$ a field and $S$ the symmetric algebra of an $N$-dimensional $K$-vector space. In other words, $S$ is a polynomial ring $K[x_1,\ldots ,x_N]$ in $N$ variables over $K$.

\subsection{Multiplier ideals}
If $K=\CC$, $\lambda\in\RR_{>0}$ and $I=(f_1,\ldots ,f_r)\subseteq S$, the {\it multiplier ideal} with coefficient $\lambda$ of $I$ is defined as
\begin{equation}\label{eq:m1}
\Jc(\lambda\b I):=\bigg\{g\in S:\frac{|g|}{(\sum_{i=1}^r|f_i|^2)^{\lambda}}\in L^1_{\mathrm{loc}}\bigg\},
\end{equation}
where $L^1_{\mathrm{loc}}$ denotes the space of locally integrable functions. This definition is quite analytic, the following definition is more geometric: If $\chara(K)=0$, $\lambda\in\RR_{>0}$ and $I$ is an ideal of $S$, the multiplier ideal with coefficient $\lambda$ of $I$ is
\begin{equation}\label{eq:m2}
\Jc(\lambda\b I):=\pi_*\OO_X(K_{X/\Spec(S)}-\lfloor \lambda\cdot F\rfloor)\footnote{While in \eqref{eq:m1} the multiplier ideal is an actual ideal of $S$, the multiplier ideal of \eqref{eq:m2} is a sheaf of ideals of $\Spec(S)$. Being $\Spec(S)$ affine, we feel free confuse the two notions.},
\end{equation}
where:
\begin{itemize}
\item[(i)] $\pi:X\longrightarrow \Spec(S)$ is a log-resolution of the sheafication $\widetilde{I}$ of $I$.
\item[(ii)] $\pi^{-1}\big(\widetilde{I}\big)=\OO_X(-F)$.
\item[(iii)] $K_{X/\Spec(S)}$ is the relative canonical divisor.
\end{itemize}
This simply means that $X$ is non-singular, $F$ is an effective divisor, the exceptional locus $E$ of $\pi$ is a divisor and $F+E$ has simple normal crossing support. Log-resolutions like this, in characteristic 0, always exist, essentially by Hironaka's celebrated result on resolution of singularities \cite{Hironaka}.

The {\it log-canonical threshold} of an ideal $I\subseteq S$ is:
\[\lct(I)=\min\{\lambda\in\RR_{>0}:\Jc(\lambda\b I)\neq S\}.\]

\subsection{Young diagrams}
A {\it (Young) diagram} is a vector $\sigma=(\sigma_1,\ldots ,\sigma_k)$ with positive integers as entries, such that $\sigma_1\geq \sigma_2\geq\ldots \geq \sigma_k\geq 1$. We say that $\sigma$ has $k$ {\it parts} and {\it height} $\sigma_1$. Given a positive integer $k$, we denote by $\P_k$ the set of diagrams with at most $k$ parts, and by $\H_k$ the set of diagrams with height at most $k$. 

The writing $\sigma=(r_1^{s_1},r_2^{s_2},\ldots )$ means that the first $s_1$ entries of $\sigma$ are equal to $r_1$, the following $s_2$ entries of $\sigma$ are equal to $r_2$ and so on... Given two diagrams $\sigma=(\sigma_1,\ldots ,\sigma_k)$ and $\tau=(\tau_1,\ldots ,\tau_h)$, for $\sigma\subseteq \tau$ we mean that $k\leq h$ and $\sigma_i\leq \tau_i$ for all $i=1,\ldots ,k$. Given a diagram $\sigma$, its {\it transpose} is 
the diagram $\tl \sigma$ given by $\tl\sigma_i=|\{j:\sigma_j\geq i\}|$.

Given a diagram $\sigma=(\sigma_1,\ldots ,\sigma_k)$, the following {\it $\gamma$-functions} will play an important role in many parts of the paper:
\begin{equation}\label{eq:gammafunctions}
\gamma_t(\sigma)=\sum_{i=1}^k\max\{0,\sigma_i-t+1\} \ \ \ \forall \ t\in\NN
\end{equation}
Given a subset of diagrams with at most $k$ parts, say $\Sigma\subseteq \H_k$, we denote by $P_{\Sigma}\subseteq \RR^k$ the convex hull of the set $\{(\gamma_1(\sigma),\gamma_2(\sigma),\ldots ,\gamma_k(\sigma)):\sigma\in\Sigma\}$. Such a polyhedron will be fundamental in our results. Notice that, if $\Sigma$ is a finite set, then $P_{\Sigma}$ is a polytope, and for the applications we are interested in we can always reduce to such a case.

Let $V$ be a $K$-vector space of dimension $n$. If $\chara(K)=0$, there is a bi-univocal correspondence between diagrams in $\P_n$ and irreducible polynomial representations of $\GL(V)$. Namely, to a diagram $\sigma$ corresponds the {\it Schur module} $S_{\sigma}V$; for example, if $\sigma =(k)$ then $S_{\sigma}V=\Sym^kV$, and if $\sigma =(1^k)$ then $S_{\sigma}V=\bigwedge^kV$.

\subsection{Representation theory and commutative algebra in $\Sym(V\otimes W^*)$}\label{sec:generic}

Let $m\leq n$ positive integers, $V$ be a $K$-vector space of dimension $m$, $W$ be a $K$-vector space of dimension $n$ and $S=\Sym(V\otimes W^*)$. On $S$ there is a natural action of the group $G=\GL(V)\times \GL(W)$ and, if $\chara(K)=0$, the \emph{Cauchy formula}
\[S=\bigoplus_{\sigma\in \P_m}S_{\sigma}V\otimes S_{\sigma}W^*\]
is the decomposition of $S$ in irreducible $G$-representations (cf. \cite[Corollary 2.3.3]{We}). Let us assume for a moment that $\chara(K)=0$. Under such an assumption, the $G$-invariant ideals of $S$ were described by DeConcini, Eisenbud and Procesi in \cite{DEP}: They are the $G$-subrepresentations of $S$ of the form
\[\bigoplus_{\sigma\in\Sigma}S_{\sigma}V\otimes S_{\sigma}W^*,\]
where $\Sigma\subseteq \P_m$ is such that $\tau\in\Sigma$ whenever there is $\sigma\in \Sigma$ with $\sigma\subseteq \tau$.
For a diagram $\sigma$ with at most $m$ parts, we will denote the ideal generated by the irreducible $G$-representation $S_{\sigma}V\otimes S_{\sigma}W^*$ by $I_{\sigma}$. Indeed,
\[I_{\sigma}=\bigoplus_{\tau\supseteq \sigma}S_{\tau}V\otimes S_{\tau}W^*.\]
More generally, for any set $\Sigma\subseteq \P_m$ let us put: 
\[I(\Sigma)=\sum_{\sigma\in\Sigma}I_{\sigma}=\bigoplus_{\substack{\tau\supseteq \sigma \\ \mbox{{\scriptsize for some }}\sigma\in\Sigma}}S_{\tau}V\otimes S_{\tau}W^*.\]

In positive characteristic, the situation is more complicated from the view-point of the action of $G$. A characteristic-free approach to the study of ``natural ideals'' in $S$ is by meaning of standard monomial theory: The ring $S$ can be seen as the polynomial ring $K[X]$ whose variables are the entries of a generic $m\times n$-matrix $X$. A distinguished ideal of such a ring is the ideal $I_t$ generated by the $t$-minors of $X$, where $t\leq m$. In characteristic 0, $I_t$ coincides with the ideal $I_{(1^t)}$. Other interesting ideals of $S$ are
\begin{equation}\label{eq:Dsigma}
D_{\sigma}=I_{\sigma_1}I_{\sigma_2}\cdots I_{\sigma_k},
\end{equation}
where $\sigma=(\sigma_1,\ldots ,\sigma_k)\in\H_m$. The integral closures of such ideals have a nice primary decomposition, with the symbolic powers of the ideals $I_t$ as primary components. As we are going to see soon, such symbolic powers are particularly easy to describe. By a {\it product of minors} we mean a product $\Pi=\delta_1\cdots \delta_k\in S$ where $\delta_i$ is a $\sigma_i$-minor of $X$ and $\sigma_1\geq \sigma_2\geq\ldots \geq \sigma_k\geq 1$. We refer to the diagram $\sigma=(\sigma_1,\ldots ,\sigma_k)$ as the  {\it shape} of $\Pi$. As shown in Theorem \ref{thm:symb}, the symbolic powers of $I_t$ are generated by product of minors of certain shapes described by the following $\gamma$-functions defined in \eqref{eq:gammafunctions}.

\begin{thm}\label{thm:symb} \cite[Theorem 7.1]{DEP}.
For any $t\leq m$ and $s\in\NN$, the symbolic power $I_t^{(s)}$ is generated by the products of minors whose shapes $\sigma$ satisfy
\[\gamma_t(\sigma)\geq s.\]
\end{thm}
So, the next result implies that to check whether a product of minors is integral over $D_{\sigma}$ is immediate.

\begin{thm} \cite[Theorem 1.3 and Remark 1.6]{Bruns} \label{thm:Dsigmadiamond}.
For a diagram $\sigma\in \H_m$, the integral closure of $D_{\sigma}$ is
\[\bigcap_{i=1}^mI_i^{(\gamma_i(\sigma))}.\]
\end{thm}

More generally, given a set $\Sigma\subseteq \H_m$, let us put $D(\Sigma)=\sum_{\sigma\in\Sigma}D_{\sigma}$.
Also for the integral closure of such ideals there is a nice description, in terms of the polyhedron $P_{\Sigma}\subseteq \RR^m$. 

\begin{thm}\label{thm:icsumsofprod}
For a subset $\Sigma\subseteq \H_m$, the integral closure of $D(\Sigma)$ is equal to
\[\sum_{\ab=(a_1,\ldots ,a_m)\in P_{\Sigma}}\left(\bigcap_{i=1}^mI_i^{(\lceil a_i\rceil)}\right).\]
\end{thm}
\begin{proof}
In characteristic 0 this has already been proved in \cite[Theorems 6.1]{DEP}. In general, the same argument used in the proof of \cite[Theorem 1.3]{Bruns} works as well as in that case.
\end{proof}

\begin{remark}
Notice that, to form the ideals $D(\Sigma)$, the set $\Sigma$ can be taken finite. Thus, in Theorem \ref{thm:icsumsofprod}, we can always let $P_{\Sigma}$ being a polytope. The analog remark holds for Theorems \ref{thm:icsumsofprodsym} and \ref{thm:icsumsofprodpf} below.
\end{remark}

When $\chara(K)=0$, the ideals $I(\Sigma)$ and $D(\Sigma)$ are related by the following:

\begin{thm}\label{thm:relinvdet}\cite[Theorems 8.1 and 8.2]{DEP}.
If $\chara(K)=0$, for any diagram $\sigma\in\P_m$ we have 
\[\overline{I_{\sigma}}=D_{\tl\sigma}.\]
In general, if $\Sigma\subseteq \P_m$, then $\overline{I(\Sigma)}=\overline{D(\tl\Sigma)}$,
where $\tl\Sigma=\{\tl\sigma:\sigma\in\Sigma\}$.
\end{thm}

\subsection{Representation theory and commutative algebra in $\Sym(\Sym^2V)$}\label{sec:symmetric}

Let $n$ be a positive integer, $V$ be a $K$-vector space of dimension $n$ and $S=\Sym(\Sym^2V)$. Let $\Rc_e$ be the set of diagrams $\sigma=(\sigma_1,\ldots ,\sigma_k)$ with $\sigma_i$ even for all $i=1,\ldots ,k$. Dually, $\Cc_e$ will be the set of diagrams $\sigma$ such that $\tl\sigma\in \Rc_e$. The general linear group $\GL(V)$ acts naturally on $S$ and, if $\chara(K)=0$,
\[S=\bigoplus_{\sigma\in \P_n\cap \Rc_e}S_{\sigma}V\]
is the decomposition of $S$ in irreducible $\GL(V)$-representations (cf. \cite[Proposition 2.3.8 (a)]{We}). Let us assume for a moment that $\chara(K)=0$. Under such an assumption, the $\GL(V)$-invariant ideals of $S$ were described by Abeasis in \cite{Ab}: They are the $\GL(V)$-subrepresentations of $S$ of the form
\[\bigoplus_{\sigma\in\Sigma}S_{\sigma}V,\]
where $\Sigma\subseteq \P_n\cap \Rc_e$ is such that 
$\tau\in\Sigma$ whenever there is $\sigma\in \Sigma$ with $\sigma\subseteq \tau$.
For a diagram $\sigma\in\P_n\cap \Rc_e$, we will denote the ideal generated by the irreducible $\GL(V)$-representation $S_{\sigma}V$ by $J_{\sigma}$. More generally, for any set $\Sigma\subseteq \P_n\cap \Rc_e$ we set: 
\[J(\Sigma)=\sum_{\sigma\in\Sigma}J_{\sigma}.\]

A characteristic-free approach to the study of commutative algebra in $S$ is, again, provided by standard monomial theory: The ring $S$ can be seen as the polynomial ring $K[Y]$ whose variables are the entries of a $n\times n$-symmetric-matrix $Y$. A distinguished ideal of such a ring is the ideal $J_{t}$ generated by the $t$-minors of $Y$, where $t\leq n$. In characteristic 0, $J_{t}$ coincides with the ideal $J_{(2^{t})}$. Other interesting ideals of $S$ are
\begin{equation}\label{eq:Esigma}
E_{\sigma}=J_{\sigma_1}J_{\sigma_2}\cdots J_{\sigma_k},
\end{equation}
where $\sigma=(\sigma_1,\ldots ,\sigma_k)\in\H_{n}$, and more generally their sums $E(\Sigma)=\sum_{\sigma\in\Sigma}E_{\sigma}$, where $\Sigma\subseteq \H_{n}$. For a product $\Pi=\delta_1\cdots \delta_k\in S$ where $\delta_i$ is a $\sigma_i$-minor of $Y$ and $\sigma_1\geq \sigma_2\geq\ldots \geq \sigma_k\geq 1$, again we refer to the diagram $\sigma=(\sigma_1,\ldots ,\sigma_k)$ as the shape of $\Pi$. 

\begin{thm} \label{thm:symbsym}\cite[Teorema 5.1]{Ab}
For any $t\leq n$ and $s\in\NN$, the symbolic power $J_{t}^{(s)}$ is generated by the products of minors whose shapes $\sigma$ satisfy
\[\gamma_t(\sigma)\geq s.\]
\end{thm}
%
%

\begin{thm}\label{thm:icsumsofprodsym}
For a subset $\Sigma\subseteq \H_{n}$, the integral closure of $E(\Sigma)$ is equal to
\[\sum_{\ab=(a_1,\ldots ,a_n)\in P_{\Sigma}}\left(\bigcap_{i=1}^nJ_{i}^{(\lceil a_i\rceil)}\right).\]
\end{thm}
\begin{proof}
In characteristic 0 this has already been proved in \cite[Teorema 4.1]{Ab}. In general, the same argument used in the proof of \cite[Theorem 1.3]{Bruns} works as well as in that case.
\end{proof}

When $\chara(K)=0$, the ideals $J(\Sigma)$ and $E(\Sigma)$ are related by the following:

\begin{thm}\cite[Teorema 6.1 and comment below]{Ab}.
If $\chara(K)=0$, for any diagram $\sigma\in\P_n\cap \Cc_e$ we have 
\[\overline{J_{\sigma}}=E_{\sigma'},\]
where $\sigma'$ is the diagram with $i$-th entry $\tl\sigma_{2i}$.
In general, if $\Sigma\subseteq \P_n\cap \Rc_e$, then $\overline{J(\Sigma)}=\overline{E(\Sigma')}$,
where $\Sigma'=\{\sigma':\sigma\in\Sigma\}$.
\end{thm}

\subsection{Representation theory and commutative algebra in $\Sym(\bigwedge^2V)$}\label{sec:skew-symmetric}

Let $n$ be a positive integer, $V$ be a $K$-vector space of dimension $n$ and $S=\Sym(\bigwedge^2V)$. The general linear group $\GL(V)$ acts naturally on $S$ and, if $\chara(K)=0$,
\[S=\bigoplus_{\sigma\in \P_n\cap \Cc_e}S_{\sigma}V\]
is the decomposition of $S$ in irreducible $\GL(V)$-representations (cf. \cite[Proposition 2.3.8 (b)]{We}). Let us assume for a moment that $\chara(K)=0$. Under such an assumption, the $\GL(V)$-invariant ideals of $S$ were described by Abeasis and Del Fra in \cite{AD}: They are the $\GL(V)$-subrepresentations of $S$ of the form
\[\bigoplus_{\sigma\in\Sigma}S_{\sigma}V,\]
where $\Sigma\subseteq \P_n\cap \Cc_e$ is such that 
$\tau\in\Sigma$ whenever there is $\sigma\in \Sigma$ with $\sigma\subseteq \tau$.
For a diagram $\sigma\in\P_n\cap \Cc_e$, we will denote the ideal generated by the irreducible $\GL(V)$-representation $S_{\sigma}V$ by $P_{\sigma}$. More generally, for any set $\Sigma\subseteq \P_n\cap \Cc_e$ we set: 
\[P(\Sigma)=\sum_{\sigma\in\Sigma}P_{\sigma}.\]

A characteristic-free approach to the study of commutative algebra in $S$ is, again, provided by standard monomial theory: The ring $S$ can be seen as the polynomial ring $K[Z]$ whose variables are the entries of a $n\times n$-skew-symmetric-matrix $Z$. A distinguished ideal of such a ring is the ideal $P_{2t}$ generated by the $2t$-Pfaffians of $Z$, where $t\leq \lfloor n/2\rfloor$. In characteristic 0, $P_{2t}$ coincides with the ideal $P_{(1^{2t})}$. Other interesting ideals of $S$ are
\begin{equation}\label{eq:Fsigma}
F_{\sigma}=P_{2\sigma_1}P_{2\sigma_2}\cdots P_{2\sigma_k},
\end{equation}
where $\sigma=(\sigma_1,\ldots ,\sigma_k)\in\H_{\lfloor n/2\rfloor}$, and more generally their sums $F(\Sigma)=\sum_{\sigma\in\Sigma}F_{\sigma}$, where $\Sigma\subseteq \H_{\lfloor n/2\rfloor}$. For a product $\Pi=\delta_1\cdots \delta_k\in S$ where $\delta_i$ is a $2\sigma_i$-Pfaffian of $Z$ and $\sigma_1\geq \sigma_2\geq\ldots \geq \sigma_k\geq 1$, we refer to the diagram $\sigma=(\sigma_1,\ldots ,\sigma_k)$ as the shape of $\Pi$. 

\begin{thm}\label{thm:symbpf}\cite[Theorem 5.1]{AD}
For any $t\leq \lfloor n/2\rfloor$ and $s\in\NN$, the symbolic power $P_{2t}^{(s)}$ is generated by the products of Pfaffians whose shapes $\sigma$ satisfy
\[\gamma_t(\sigma)\geq s.\]
\end{thm}

\begin{thm}\label{thm:icsumsofprodpf}
For a subset $\Sigma\subseteq \H_{\lfloor n/2\rfloor}$, the integral closure of $F(\Sigma)$ is equal to
\[\sum_{\ab=(a_1,\ldots ,a_{\lfloor n/2\rfloor})\in P_{\Sigma}}\left(\bigcap_{i=1}^{\lfloor n/2\rfloor}P_{2i}^{(\lceil a_i\rceil)}\right).\]
\end{thm}
\begin{proof}
In characteristic 0 this has already been proved in \cite[Theorem 4.1]{AD}. In general, similar arguments to those used in the proof of \cite[Theorem 1.3]{Bruns} work.
\end{proof}

When $\chara(K)=0$, the ideals $P(\Sigma)$ and $F(\Sigma)$ are related by the following:

\begin{thm}\cite[Theorems 6.1 and 6.2]{AD}.
If $\chara(K)=0$, for any diagram $\sigma\in\P_n\cap \Cc_e$ we have 
\[\overline{P_{\sigma}}=D_{\tilde{\sigma}},\]
where by $\tilde{\sigma}$ we mean the diagram with $i$-th entry $\tl\sigma_i/2$.
In general, if $\Sigma\subseteq \P_n\cap \Cc_e$, then $\overline{P(\Sigma)}=\overline{F(\widetilde{\Sigma})}$,
where $\widetilde{\Sigma}=\{\widetilde{\sigma}:\sigma\in\Sigma\}$.
\end{thm}

\subsection{$F$-pure threshold and test ideals}

In this subsection $\chara(K)=p>0$. Given an ideal $I=(f_1,\ldots ,f_r)$ of $S$ and a power of $p$, say $q=p^e$, the {\it $q$-th Frobenius power} of $I$ is:
\[I^{[q]}=(f_1^q,\ldots ,f_r^q)=(f^q:f\in I).\] 
Let $\mm$ denote the irrelevant ideal of $S$ and consider a homogeneous ideal $I$. For any $q=p^e$, define the function:
\[\nu_I(q):=\max\{r:I^r\nsubseteq \mm^{[q]}\}.\]
The {\it F-pure threshold} of $I$ (at $\mm$) is defined as
\[\fpt(I):=\lim_{e\to\infty}\frac{\nu_I(q)}{q}.\]
In Proposition \ref{prop:rangefpt}, we will point out a (sharp) range in which $\fpt(I)$ can vary. While the upper bound is well known, the lower bound is less popular. Let $d(I)$ be the largest degree of a minimal generator of $I$. Also, we set
\[\delta(I):=\lim_{k\to\infty}\frac{d(I^k)}{k}.\]
Notice that $\delta(I)\leq d(I)$ and that $\delta(I)=d(I)$ if all the minimal generators of $I$ have degree $d(I)$. The following proof is based on the fact that $\nu_I(q)=\nu_{g(I)}(q)$ for any linear homogeneous change of coordinates $g$ on $S$, because $\mm^{[q]}=(x_1^q,\ldots ,x_N^q)=(g(x_1)^q,\ldots ,g(x_N)^q)$.

\begin{prop}\label{prop:rangefpt}
If $\chara(K)=p>0$, then any homogeneous ideal $I\subseteq S$ satisfies the inequalities:
\[\frac{\height(I)}{\delta(I)}\leq \fpt(I)\leq \height(I).\]
\end{prop}
\begin{proof}
To show the inequality $\fpt(I)\leq \height(I)$ notice that, by the pigeonhole principle, because $S_{\pp}$ is a regular local ring of dimension $\height (\pp)$ for all $\pp\in\Spec(S)$, for all positive integers $r$ we have 
\[\pp_{\pp}^r\subseteq \pp_{\pp}^{[q]} \mbox{ whenever $q=p^e$ and }r>(q-1)\height (\pp).\]
Intersecting back with $S$, by the flatness of the Frobenius, we get $\pp^{(r)}\subseteq \pp^{[q]}$ whenever $r>(q-1)\height (\pp)$. This gives the desired inequality by taking as $\pp$ a minimal prime of $I$ of the same height of $I$.

For the inequality $\fpt(I)\geq \height(I)/\delta(I)$, recall that, as proved in \cite{CHT} and in \cite{Ko}, there exists $\alpha(I)$ such that
\[\reg(I^k)=\delta(I)\cdot k+\alpha(I) \ \ \ \forall \ k\gg 0.\]
Let us consider the generic initial ideal w.r.t. the degrevlex term order, $\gin(I^k)$. By the main result in \cite{BS}, $\reg(\gin(I^k))=\reg(I^k)$. If $k$ is large enough, then $\gin(I^k)$ is a Borel-fixed ideal of regularity $\delta(I)\cdot k+\alpha(I)=:r(k)$. Therefore, by \cite[Proposition 10]{ERT}
\[\gin(I^k)_{\geq r(k)}\]
is a stable ideal. If $c=\height(I)=\height(\gin(I^k))$, thus $x_c^{r(k)}\in \gin(I^k)_{\geq r(k)}$. By the stability of $\gin(I^k)_{\geq r(k)}$ this implies that
\[u(k):=x_1^{\lceil r(k)/c\rceil}\cdots x_c^{\lceil r(k)/c\rceil}\in \gin(I^k)_{\geq r(k)}\subseteq \gin(I^k).\]
Pick a linear homogeneous change of coordinates $g$ such that $\gin(I^k)=\init(g(I^k))$. In particular for $q=p^e$ we have
\[u(k)^{\big\lceil \frac{q}{\lceil r(k)/c\rceil}\big\rceil -1}\in \init\bigg(g(I)^{k\big( \big\lceil \frac{q}{\lceil r(k)/c\rceil}\big\rceil -1\big)}\bigg)\setminus \mm^{[q]},\]
from which
\[\nu_I(q)=\nu_{g(I)}(q)\geq \frac{kq}{r(k)/c+1}-k.\]
If $q\gg k\gg 0$, the asymptotic of the above quantity is $cq/\delta(I)$, and this lets us conclude.
\end{proof}

\begin{remark}
When $I$ is generated in a single degree, the above lower bound has been shown in \cite[Proprosition 4.2]{TW}. A (more powerful) variant for the log-canonical threshold is in \cite[Theorem 3.4]{dEM}.
\end{remark}

Given any ideal $I\subseteq S$ and $q=p^e$, the {\it $q$-th root of $I$}, denoted by $I^{[1/q]}$, is the smallest ideal $J\subseteq S$ such that $I\subseteq J^{[q]}$. By the flatness of the Frobenius over $S$ the $q$-th root is well defined. Let $I$ be an ideal of $S$ and $\lambda$ be a positive real number. It is easy to see that
\[\bigg(I^{\lceil \lambda p^e\rceil}\bigg)^{[1/p^e]}\subseteq \bigg(I^{\lceil \lambda p^{e+1}\rceil}\bigg)^{[1/p^{e+1}]}.\]
The {\it test ideal} of $I$ with coefficient $\lambda$ is defined as:
\[\tau(\lambda\b I):=\bigcup_{e>0}\bigg(I^{\lceil \lambda p^e\rceil}\bigg)^{[1/p^e]}\underset{e\gg 0}{=}\bigg(I^{\lceil \lambda p^e\rceil}\bigg)^{[1/p^e]}.\]
For any ideal $I\subseteq S$, we can therefore define the $F$-pure threshold (consistently with what we hade done in the homogeneous case) as:
\[\fpt(I)=\min\{\lambda \in\RR_{>0}:\tau(\lambda\b I)\neq S\}.\]

If $\lambda\in\RR_+$ and $I$ is an ideal in a polynomial ring over a field of characteristic 0, denoting by $_p$ the reduction modulo the prime number $p$, Hara and Yoshida proved  in \cite[Theorem 6.8]{Hara-Yoshida} that:
\begin{equation}\label{eq:HY}
\Jc(\lambda\b I)_p=\tau(\lambda\b I_p)
\end{equation}
for all $p\gg 0$ (depending on $\lambda$). In particular,
\[\lim_{p\to\infty}\fpt(I_p)=\lct(I).\]

The following lemma will be useful to the proof of Proposition \ref{prop:generalinclusion}.

\begin{lemma}\label{lem:maxreg}
Let $R$ be a Noetherian commutative ring of positive characteristic, and $I=(r_1,\ldots ,r_s)\subseteq R$ be an ideal. If the local cohomology module $H_I^s(R)$ is not zero, then there exist ideals $J_{\lambda}\supsetneq I$ such that
\[\tau(\lambda\b I)=J_{\lambda}I^{\lfloor \lambda\rfloor -s} \ \ \ \forall \ \lambda \geq s.\]
In particular, if $(R,\mm)$ is a $d$-dimensional regular local ring of positive characteristic, then
\[\tau(\lambda\b \mm)=\mm^{\lfloor \lambda\rfloor +1-d}\]
\end{lemma}
\begin{proof}
Skoda's theorem (cf. \cite[Proposition 2.25]{BMS:MMJ}) implies that, whenever $\lambda \geq s$, 
\[\tau(\lambda\b I)=I\cdot\tau((\lambda-1)\b I).\]
So it is enough to show that $\tau(\lambda\b I)\supsetneq I$ whenever $\lambda <s$. To see this, let us set $r:=r_1\cdots r_s$. By the equivalence between local and \u Cech cohomology, it is not difficult to see that $H_I^s(R)\neq 0$ if and only if there exists $a>0$ such that $r^{q-a}\notin (r_1^q,\ldots ,r_s^q)$ for any $q\geq a$. So, if $q$ is a power of the characteristic of $R$,
\[r^{q-a}\in I^{s(q-a)}\setminus I^{[q]}\ \ \ \forall \ q\geq a,\]
which implies that $\tau(\lambda\b I)\supsetneq I$ whenever $\lambda <s$.
\end{proof}
%

\section{Floating test ideals}\label{sec:floating}

Let $K$ be a field, and $S=K[x_1,\ldots ,x_N]$ be the polynomial ring in $N$ variables over $K$. For an ideal $I\subseteq S$ and a prime ideal $\pp\subseteq S$, we define the function $f_{I:\pp}:\ZZ_{>0}\longrightarrow \ZZ_{>0}$ as:
\[f_{I;\pp}(s)=\max\{\ell:I^s\subseteq \pp^{(\ell)}\} \ \ \ \forall \ s\in\ZZ_{>0}.\]

\begin{lemma}
The function above is linear. That is, $f_{I;\pp}(s)=f_{I;\pp}(1)\cdot s$ for any positive integer $s$.
\end{lemma}
\begin{proof}
By definition of symbolic power, $I^s\subseteq \pp^{(\ell)} \iff I_{\pp}^s\subseteq \pp_{\pp}^{\ell}$ in $S_{\pp}$. Obviously, $I_{\pp}\subseteq \pp_{\pp}^{\ell}$ implies that $I_{\pp}^s\subseteq \pp_{\pp}^{s\ell}$, which yields $f_{I;\pp}(s)\geq f_{I;\pp}(1)\cdot s$. For the other inequality, take $x\in I_{\pp}\setminus \pp_{\pp}^{\ell +1}$. Then $\overline{x}$ is a nonzero element of degree $\ell$ in $R=\mathrm{gr}_{\pp_{\pp}}(S_{\pp})$. Since $S$ is regular, $R$ is a polynomial ring. In particular it is reduced, thus $\overline{x}^s$ is a nonzero element of degree $\ell s$ in $R$. So $x^s\in I_{\pp}^s\setminus \pp_\pp^{\ell s+1}$, which implies $f_{I;\pp}(s)\leq f_{I;\pp}(1)\cdot s$.
\end{proof}

From now on, for an ideal $I\subseteq S$ and a prime ideal $\pp\subseteq S$, we introduce the notation:
\begin{equation}\label{eq:epi}
e_{\pp}(I):=f_{I;\pp}(1)=\max\{\ell:I\subseteq \pp^{(\ell)}\}.
\end{equation}

\begin{prop}\label{prop:generalinclusion}
If $K$ has positive characteristic and $I\subseteq S$ is an ideal, then
\[\tau(\lambda\b I)\subseteq\bigcap_{\substack{\pp\in\Spec(S) \\ 
\pp\supseteq I}}\pp^{(\lfloor \lambda e_{\pp}(I)\rfloor +1-\height(\pp))} \ \ \ \forall \ \lambda\in \RR_{>0}.\]
\end{prop}
\begin{proof}
Let us fix $\lambda\in\RR_{>0}$. For any prime ideal $\pp\supseteq I$, we need to show that 
\[I^{\lceil \lambda q\rceil} \subseteq \bigg(\pp^{(\lfloor \lambda e_{\pp}(I)\rfloor +1-\height(\pp))}\bigg)^{[q]} \ \ \ \mbox{for } \ q\gg 0,\]
where $q$ is a power of $\chara(K)=p$. To see this, let us take $q=p^e$ and start with the inclusion:
\[I^{\lceil \lambda q\rceil} \subseteq \pp^{(\lceil \lambda q\rceil e_{\pp}(I))}.\]
By localizing at $\pp$, we have
\[(I^{\lceil \lambda q\rceil})S_{\pp} \subseteq (\pp S_{\pp})^{\lceil \lambda q\rceil e_{\pp}(I)}.\]
Because $S_{\pp}$ is a regular local ring of dimension $\height(\pp)$, by using Lemma \ref{lem:maxreg} we infer that
\begin{eqnarray*}
(\pp S_{\pp})^{\lceil \lambda q\rceil e_{\pp}(I)} & \subseteq & (\pp S_{\pp})^{\lceil \lambda e_{\pp}(I)q\rceil} \\ 
& \underset{q\gg 0}{\subseteq} & \bigg((\pp S_{\pp})^{\lfloor \lambda e_{\pp}(I)\rfloor +1-\height(\pp)}\bigg)^{[q]}
\end{eqnarray*}
So, when $q\gg 0$ we obtain that:
\[(I^{\lceil \lambda q\rceil})S_{\pp} \subseteq \bigg((\pp S_{\pp})^{\lfloor \lambda e_{\pp}(I)\rfloor +1-\height(\pp)}\bigg)^{[q]}.\]
By the flatness of the Frobenius over $S$, by intersecting back with $S$ we get:
\[I^{\lceil \lambda q\rceil} \subseteq \bigg(\pp^{(\lfloor \lambda e_{\pp}(I)\rfloor +1-\height(\pp))}\bigg)^{[q]},\]
which is what we wanted.
\end{proof}

\begin{definition}\label{def:floating}
We will say that an ideal $I\subseteq S$ has {\it floating test ideals} if the inclusion in Proposition \ref{prop:generalinclusion} is an equality for all $\lambda\in \RR_{>0}$.
\end{definition}

Below, we will introduce a class of ideals with floating test ideals. Such ideals have properties quite combinatorial in nature: as we will see, in the class lie all the ideals $D_{\sigma}$, $E_{\sigma}$ and $F_{\sigma}$ introduced in Section \ref{sec:basics}. Before stating the definition, let us observe that, if the inclusion
\[I^s\subseteq \bigcap_{\substack{\pp\in\Spec(S) \\ 
\pp\supseteq I}}\pp^{(f_{I;\pp}(s))}\]
happens to be an equality, then $I^s$ must be integrally closed: indeed, symbolic powers of prime ideals in a regular ring are integrally closed, and the intersection of integrally closed ideals is obviously integrally closed. Furthermore, recall that Ratliff proved in \cite[Theorems (2.4)]{Ra} that 
\[\Ass\big(\overline{I^s}\big)\subseteq \Ass\big(\overline{I^{s+1}}\big) \ \ \ \forall \ s\in\ZZ_{>0},\]
and in \cite[Theorems (2.7)]{Ra} that 
\[\bigg|\bigcup_{s\in \ZZ_{>0}}\Ass\big(\overline{I^s}\big)\bigg|<+\infty.\]
Let us denote by $\StAss(I)=\bigcup_{s\in \ZZ_{>0}}\Ass\big(\overline{I^s}\big)$ and introduce the following central definition:

\begin{definition}[{\bf Condition ($\diamond$)}]\label{def:diamond}
An ideal $I\subseteq S$ satisfies condition ($\diamond$) if, for any $s\gg 0$, there exists a primary decomposition of $\overline{I^s}$ consisting of symbolic powers of the prime ideals in $\StAss(I)$. In other words, there exist functions $g_{I;\pp}:\NN \rightarrow\NN$ such that:
\begin{equation}\label{eq:diamond}
\overline{I^s}=\bigcap_{\pp\in\StAss(I)}\pp^{(g_{I;\pp}(s))} \ \ \ \forall \ s\gg 0.
\end{equation}
\end{definition}

The functions $g_{I;\pp}$ may not be linear, however the next lemma shows that such a failure is paltry enough.

\begin{lemma}
Let $I\subseteq S$ be an ideal satisfying condition ($\diamond$) generated by $\mu$ elements. Then, for all $\pp\in \StAss(I)$, there exist a function $r_{I;\pp}:\NN\rightarrow \NN$ such that $0\leq r_{I,\pp}(s)\leq e_{\pp}(I)(\mu -1)$ and
\begin{equation}\label{eq:diamondlinear}
\overline{I^s}=\bigcap_{\pp\in\StAss(I)}\pp^{(e_{\pp}(I)s-r_{I;\pp}(s))} \ \ \ \forall \ s\gg 0,
\end{equation}
where the $e_{\pp}(I)$'s have been defined in \eqref{eq:epi}.
\end{lemma}
\begin{proof}
For all positive integer $s$, we have
\[g_{I;\pp}(S)=\max\{\ell:\overline{I^s}\subseteq \pp^{(\ell)}\}\leq f_{I;\pp}(s)=e_{\pp}(I)s.\]
On the other hand, Brian\c con-Skoda theorem implies that
\[\overline{I^{s+\mu -1}}\subseteq I^s \ \ \ \forall s\in\ZZ_{>0}.\]
Therefore: 
\[g_{I;\pp}(s)\geq f_{I;\pp}(s-\mu +1)=e_{\pp}(I)s-e_{\pp}(I)(\mu -1) \forall \ s\geq \mu.\] 
The existence of $r_{I;\pp}$ follows at once.
\end{proof}

Let us give some examples of ideals satisfying condition ($\diamond$).

\begin{example}\label{ex:ci}
Any prime ideal $\pp\subseteq S$ which is a complete intersection, satisfies condition ($\diamond$). Indeed, since $S/\pp^s$ is Cohen-Macaulay for all $s>0$, $\pp^s=\pp^{(s)}$ in this case.
\end{example}

\begin{example}\label{ex:generic} The ideals $D_{\sigma}$ defined in \eqref{eq:Dsigma} satisfy condition ($\diamond$): indeed, Theorem \ref{thm:Dsigmadiamond} implies that
\[\overline{D_{\sigma}^s}=\bigcap_{i=1}^{m}I_{i}^{(s\gamma_i(\sigma))}.\]
\end{example}

\begin{example}\label{ex:symmetric} The ideals $E_{\sigma}$ defined in \eqref{eq:Esigma} satisfy condition ($\diamond$): indeed, in such a case $\Sigma = \{\sigma\}$, therefore Theorem \ref{thm:icsumsofprodsym} implies that
\[\overline{E_{\sigma}^s}=\bigcap_{i=1}^{n}J_{i}^{(s\gamma_i(\sigma))}.\]
\end{example}

\begin{example}\label{ex:skew-symmetric} The ideals $F_{\sigma}$ defined in \eqref{eq:Fsigma} satisfy condition ($\diamond$): indeed, once again $\Sigma = \{\sigma\}$, therefore Theorem \ref{thm:icsumsofprodpf} implies that
\[\overline{F_{\sigma}^s}=\bigcap_{i=1}^{\lfloor n/2\rfloor}P_{2i}^{(s\gamma_i(\sigma))}.\]
\end{example}

Condition ($\diamond$) alone is not enough to guarantee the equality in Proposition \ref{prop:generalinclusion}, as it is evident from Example \ref{ex:ci}. So, we introduce another central definition:

\begin{definition}[{\bf Condition ($\diamond +$)}]\label{def:diamond+}
An ideal $I\subseteq S$ satisfies condition ($\diamond +$) if it satisfies condition ($\diamond$) and there exists a term ordering $\prec$ on $S$ and a polynomial $F\in S$ such that:
\begin{itemize}
\item[(i)] $\init_{\prec}(F)$ is a square-free monomial;
\item[(ii)] $F\in \pp^{(\height(\pp))}$ for all $\pp\in\StAss(I)$.
\end{itemize}
\end{definition}

Before proving the next result, let us see that the ideals in Examples \ref{ex:generic}, \ref{ex:symmetric} and \ref{ex:skew-symmetric} satisfy condition ($\diamond +$). 

\vspace{3mm}

\begin{example} \label{ex:generic+} Let us consider $\Delta$ to be the product of all principal minors of $X$:
\begin{equation}\label{eq:deltagen}
\Delta := \prod_{i=0}^{n-m}[1,\ldots , m|i+1,\ldots ,i+m]\cdot \prod_{i=1}^{m-1}[i+1,\ldots , m|1,\ldots ,m-i][1,\ldots ,m-i|i+n-m,\ldots ,n].
\end{equation}
By considering the lexicographical term ordering $\prec$ extending the linear order
\[x_{11}>x_{12}>\ldots x_{1n}>x_{21}>\ldots >x_{2n}>\ldots >x_{mn},\]
we have that 
\[\init_{\prec}(\Delta)=\prod_{\substack{ i\in\{1,\ldots ,m\} \\ j\in\{1,\ldots ,n\}}}x_{ij}\] 
is a square-free monomial. Let $\tau$ be the shape of $\Delta$, namely $\tau=(m^{n-m+1},(m-1)^2,\ldots ,1^2)$ and notice that, for all $t\in\{1,\ldots ,m\}$,
\[\gamma_t(\tau)=(n-m+1)(m-t+1)+2\sum_{j=1}^{m-t}j=(n-m+1)(m-t+1)+(m-t)(m-t+1)=(n-t+1)(m-t+1).\]
Since $\height(I_t)=(n-t+1)(m-t+1)$, by Theorem \ref{thm:symb} $\Delta\in I_t^{(\height(I_t))}$ for all $t\in\{1,\ldots ,m\}$. By exploiting Example \ref{ex:generic}, thus, the ideals $D_{\sigma}$ introduced in \eqref{eq:Dsigma} satisfy condition ($\diamond +$).
\end{example}

\begin{example} \label{ex:symmetric+} Let us consider $\Delta$ to be the product of all principal upper diagonal minors of $Y$:
\[\Delta := \prod_{i=1}^{n}[1,\ldots , n-i+1|i,\ldots ,n].\]
By considering the lexicographical term ordering $\prec$ extending the linear order
\[y_{11}>y_{12}>\ldots y_{1n}>y_{22}>\ldots >y_{2n}>\ldots >y_{nn},\]
we have that 
\[\init_{\prec}(\Delta)=\prod_{1\leq i\leq j \leq n}y_{ij}\] 
is a square-free monomial. Let $\tau$ be the shape of $\Delta$, namely $\tau=(n,n-1,\ldots ,2,1)$, and notice that, for all $t\in\{1,\ldots ,n\}$,
\[\gamma_t(\tau)=\sum_{j=1}^{n-t+1}j=\binom{n-t+2}{2}.\]
Since $\height(J_t)=\binom{n-t+2}{2}$, by Theorem \ref{thm:symbsym} $\Delta\in J_t^{(\height(J_t))}$ for all $t\in\{1,\ldots ,n\}$. By exploiting Example \ref{ex:symmetric}, thus, the ideals $E_{\sigma}$ introduced in \eqref{eq:Esigma} satisfy condition ($\diamond +$).
\end{example}

\begin{example} \label{ex:skew-symmetric+}. Let us consider $\Delta$ to be the following product of Pfaffians of $Z$:
{\scriptsize \[\Delta := \begin{cases}\displaystyle [1,\ldots ,n-1][2,\ldots ,n][1,\ldots ,\widehat{\frac{n+1}{2}},\ldots ,n]\prod_{i=1}^{(n-1)/2-1}[1,\ldots ,2i][1,\ldots , \widehat{{\scriptsize i+1}},\ldots , 2i+1][n-2i,\ldots ,\widehat{{\scriptsize n-i}},\ldots ,n][n-2i+1,\ldots , n] & \text { if $n$ is odd} \\
[1,\ldots, n]\displaystyle\prod_{i=1}^{n/2-1}[1,\ldots ,2i][1,\ldots , \widehat{i+1},\ldots , 2i+1][n-2i,\ldots ,\widehat{n-i},\ldots ,n][n-2i+1,\ldots , n] & \text { if $n$ is even}
\end{cases}\]}
By considering the lexicographical term ordering $\prec$ extending the linear order
\[z_{1n}>\ldots > z_{12}>z_{2n}>\ldots >z_{23}>\ldots >z_{n-1n},\]
we have that 
\[\init_{\prec}(\Delta)=\prod_{1\leq i< j\leq n}z_{ij}\]
is a square-free monomial. Let $\tau_o$ (resp. $\tau_e$) be the shape of $\Delta$ if $n$ is odd (resp. $n$ is even); that is, $\tau_o=((\frac{n-1}{2})^3,(\frac{n-1}{2}-1)^4,\ldots ,1^4)$ and $\tau_e=(n/2,(n/2-1)^4,(n/2-2)^4,\ldots ,1^4)$. Notice that, for all $t\in\{1,\ldots ,\lfloor n/2\rfloor\}$,
\begin{eqnarray*}
\gamma_t(\tau_o) = & 3\bigg(\frac{n-1}{2}-t+1\bigg)+4\cdot \sum_{j=1}^{\frac{n-1}{2}-t}j=\bigg(\frac{n-1}{2}-t+1\bigg)(n-2t+2) = &  \\ 
& = (n/2-t+1)(n-2t+1)=(n/2-t+1)+4\cdot \sum_{j=1}^{n/2-t}j & = \gamma_t(\tau_e).
\end{eqnarray*}
Since $\height(P_{2t})=(n/2-t+1)(n-2t+1)$, by Theorem \ref{thm:symbpf} $\Delta\in P_{2t}^{(\height(P_{2t}))}$ for all $t\in\{1,\ldots ,\lfloor n/2\rfloor\}$. By exploiting Example \ref{ex:skew-symmetric}, thus, the ideals $F_{\sigma}$ introduced in \eqref{eq:Fsigma} satisfy condition ($\diamond +$).
\end{example}

\begin{thm}\label{thm:main}
If $K$ has positive characteristic and $I\subseteq S$ is an ideal enjoying the condition ($\diamond +$), then it has floating test ideals. In other words:
\[\tau(\lambda\b I)=\bigcap_{\pp\in\StAss(I)}\pp^{(\lfloor \lambda e_{\pp}(I)\rfloor +1-\height(\pp))} \ \ \ \forall \ \lambda\in \RR_{>0}.\]
In particular (independently on the characteristic!):
\[\fpt(I)=\min_{\pp\in\StAss(I)}\{\height(\pp)/e_{\pp}(I)\}\]
\end{thm}
\begin{proof}
Fix $\lambda\in\RR_{>0}$. By Proposition \ref{prop:generalinclusion}, we already know that 
\[\tau(\lambda\b I)\subseteq \bigcap_{\pp\in\StAss(I)}\pp^{(\lfloor \lambda e_{\pp}(I)\rfloor +1-\height(\pp))},\]
so we will focus on the other inclusion. Take 
\[f\in \bigcap_{\pp\in\StAss(I)}\pp^{(\lfloor \lambda e_{\pp}(I)\rfloor +1-\height(\pp))}.\] 
Consider $F$ and $\prec$ as in the definition of the condition ($\diamond +$). For any $\pp\in\StAss(I)$ and $q=p^e$ (where $\chara(K)=p$), notice that:
\begin{eqnarray*}
F^{q-1}\cdot f^q & \in & \bigg(\pp^{(\height(\pp)}\bigg)^{q-1} \cdot \bigg(\pp^{(\lfloor \lambda e_{\pp}(I)\rfloor +1-\height(\pp))}\bigg)^{q} \\
& \subseteq & \pp^{((q-1)\height(\pp)+q(\lfloor \lambda e_{\pp}(I)\rfloor +1-\height(\pp)))} \\
& = & \pp^{(q\lfloor \lambda e_{\pp}(I)\rfloor +q-\height(\pp))} \\
& = & \pp^{\bigg(q\bigg(\lfloor \lambda e_{\pp}(I)\rfloor +\frac{q-\height{\pp}}{q}\bigg)\bigg)}.
\end{eqnarray*}
If $q$ is big enough, then 
\[q\bigg(\lfloor \lambda e_{\pp}(I)\rfloor +\frac{q-\height{\pp}}{q}\bigg)\geq q\lambda e_{\pp}(I).\]
By \cite[Proposition 2.14]{BMS:MMJ}, we can assume that $q\lambda$ is an integer, and so we will do from now on.
So, let us fix $q$ big enough so that 
\[F^{q-1}f^q\in \pp^{(q\lambda e_{\pp}(I))} \ \ \ \forall \ \pp\in\StAss(I).\] 
Therefore,
\[F^{q-1}f^q\in \overline{I^{q\lambda}}.\]
Take a positive integer $k$ such that 
\[\bigg(\overline{I^{q\lambda}}\bigg)^{k+\ell}\subseteq I^{q\ell\lambda}  \ \ \ \forall \ \ell\in \NN.\]
In particular, if $q'$ a power of $p$, we have
\[F^{(q-1)(q'+k)}f^{q(q'+k)}\in I^{qq'\lambda}.\]
Let $\B_{qq'}$ be the basis of $S$ over $S^{qq'}$ consisting in monomials. Remembering that $q$ has been fixed, we can choose $q'$ big enough such that 
\[v:=\init_{\prec}(F^{(q-1)(q'+k)}f^{qk})=\init_{\prec}(F)^{(q-1)(q'+k)}\init_{\prec}(f)^{qk}\in\B_{qq'}.\]
In fact, it is enough to take $q'>qk(\deg(f)+1)$. So
\[F^{(q-1)(q'+k)}f^{qk}=v+\sum_{\substack{u\in\B_{qq'} \\ u\prec v}}g_u^{qq'}u.\]
Therefore, we get
\[F^{(q-1)(q'+k)}f^{q(q'+k)}=f^{qq'}v+\sum_{\substack{u\in\B_{qq'} \\ u\prec v}}(f g_u)^{qq'}u,\]
from which we deduce that $f\in (I^{\lceil qq'\lambda \rceil})^{[1/qq']}$ by using \cite[Proposition 2.5]{BMS:MMJ}. So
\[f\in \tau(\lambda\b I).\]
\end{proof}

An important consequence of Theorem \ref{thm:main}, together with Examples \ref{ex:generic+}, \ref{ex:symmetric+} and \ref{ex:skew-symmetric+}, is that the products of determinantal (or Pfaffian) ideals have floating test ideals. Moreover, we have the following explicit formulas for their generalized test ideals:

\begin{corollary}\label{cor:generic}
With the notation of \ref{sec:generic}, $D_{\sigma}$ has floating test ideals $\forall \ \sigma\in \H_m$. Precisely:
\[\tau\big(\lambda\b D_{\sigma}\big)=\bigcap_{i=1}^mI_i^{(\lfloor \lambda \gamma_i(\sigma)\rfloor +1-(m-i+1)(n-i+1))} \ \ \ \forall \ \lambda\in \RR_{>0}.\]
Equivalently, $\tau\big(\lambda\b D_{\sigma}\big)$ is generated by the products of minors of $X$ whose shape $\alpha$ satisfies:
\[\gamma_i(\alpha)\geq \lfloor\lambda \gamma_i(\sigma)\rfloor +1-(m-i+1)(n-i+1) \ \ \forall \ i=1,\ldots ,m.\]
In particular (independently on the characteristic!):
\[\fpt(D_{\sigma})=\min\bigg\{\frac{(m-i+1)(n-i+1)}{\gamma_i(\sigma)}:i=1,\ldots ,m\bigg\}.\]
\end{corollary}

\begin{corollary}\label{cor:symmetric}
With the notation of \ref{sec:symmetric}, $E_{\sigma}$ has floating test ideals $\forall \ \sigma\in \H_n$. Precisely:
\[\tau\big(\lambda\b E_{\sigma}\big)=\bigcap_{i=1}^nJ_i^{(\lfloor \lambda \gamma_i(\sigma)\rfloor +1-\binom{n-i+2}{2})} \ \ \ \forall \ \lambda\in \RR_{>0}.\]
Equivalently, $\tau\big(\lambda\b E_{\sigma}\big)$ is generated by the products of minors of $Y$ whose shape $\alpha$ satisfies:
\[\gamma_i(\alpha)\geq \lfloor\lambda \gamma_i(\sigma)\rfloor +1-\binom{n-i+2}{2} \ \ \forall \ i=1,\ldots ,n.\]
In particular (independently on the characteristic!):
\[\fpt(E_{\sigma})=\min\bigg\{\frac{\binom{n-i+2}{2}}{\gamma_i(\sigma)}:i=1,\ldots ,n\bigg\}.\]
\end{corollary}

\begin{corollary}\label{cor:skew-symmetric}
With the notation of \ref{sec:skew-symmetric}, $F_{\sigma}$ has floating test ideals $\forall \ \sigma\in \H_{\lfloor n/2\rfloor}$. Precisely:
\[\tau\big(\lambda\b F_{\sigma}\big)=\bigcap_{i=1}^{\lfloor n/2\rfloor}P_{2i}^{(\lfloor \lambda \gamma_i(\sigma))\rfloor +1-(n/2-i+1)(n-2i+1))} \ \ \ \forall \ \lambda\in \RR_{>0}.\]
Equivalently, $\tau\big(\lambda\b F_{\sigma}\big)$ is generated by the products of Pfaffians of $Z$ whose shape $\alpha$ satisfies:
\[\gamma_i(\alpha)\geq \lfloor\lambda \gamma_i(\sigma)\rfloor +1-(n/2-i+1)(n-2i+1) \ \ \forall \ i=1,\ldots ,\lfloor n/2\rfloor.\]
In particular (independently on the characteristic!):
\[\fpt(F_{\sigma})=\min\bigg\{\frac{(n/2-i+1)(n-2i+1)}{\gamma_i(\sigma)}:i=1,\ldots ,\lfloor n/2\rfloor\bigg\}.\]
\end{corollary}

\section{Multiplier ideals of $G$-invariant ideals}\label{sec:mult}
The goal of this section is to give explicit formulas for the multipliers ideal of all the $G$-invariant ideals in the following polynomial rings $S$ over a field of characteristic 0:

\begin{itemize}
\item[(i)] $S=\Sym(V\otimes W^*)$, where $V$ and $W$ are finite $K$-vector spaces, $G=\GL(V)\times \GL(W)$ and the action extends the diagonal one on $V\otimes W^*$.
\item[(ii)] $S=\Sym(\Sym^2V)$, where $V$ is a finite $K$-vector spaces, $G=\GL(V)$ and the action extends the natural one on $\Sym^2V$. 
\item[(iii)] $S=\Sym(\bigwedge^2V)$, where $V$ is a finite $K$-vector spaces, $G=\GL(V)$ and the action extends the natural one on $\bigwedge^2V$. 
\end{itemize} 

In order to do this, we will compute suitable generalized test ideals in positive characteristic. We need the following variant of the condition ($\diamond+$).

\begin{definition}[{\bf Condition ($\ast$)}]\label{def:ast}
An ideal $I\subseteq S$ satisfies condition ($\ast$) if there are prime ideals $\pp_1,\ldots ,\pp_k$ and a polytope $P\subseteq \RR^k$ such that:
\begin{equation}\label{eq:ast}
\overline{I^s}=\sum_{\ab=(a_1,\ldots ,a_k)\in P}\left(\bigcap_{i=1}^k\pp_i^{(\lceil sa_i\rceil)}\right) \ \ \ \forall \ s\gg 0,
\end{equation}
and there exists a term ordering $\prec$ on $S$ and a polynomial $F\in S$ such that:
\begin{itemize}
\item[(i)] $\init_{\prec}(F)$ is a square-free monomial;
\item[(ii)] $F\in \pp_i^{(\height(\pp_i))}$ for all $i=1,\ldots ,k$.
\end{itemize}
\end{definition}

\begin{example}\label{ex:detast}
Given two diagrams $\sigma=(\sigma_1,\ldots ,\sigma_k)$ and $\tau=(\tau_1,\ldots ,\tau_h)$ let us denote by $\sigma*\tau$ their concatenation $(\sigma_1,\ldots ,\sigma_k,\tau_1,\ldots ,\tau_h)$ with the entries rearranged decreasingly (so that $\sigma*\tau$ is a diagram). For a set $\Sigma$ of diagrams and $s\in\NN$, let us introduce the notation
\[\Sigma^s:=\{\sigma^{(i_1)}* \cdots *\sigma^{(i_s)}:\sigma^{(i_j)}\in\Sigma\}.\] 
Notice that, if $\Sigma\subseteq \H_k$ for some $k\in\NN$, the convex set $P_{\Sigma^s}\subseteq \RR^k$ is nothing but $s\cdot P_{\Sigma}$.
Therefore, Theorem \ref{thm:icsumsofprod} implies that,
for a subset $\Sigma\subseteq \H_m$, the integral closure of $D(\Sigma)^s=D(\Sigma^s)$ is equal to
\[\sum_{\ab=(a_1,\ldots ,a_m)\in P_{\Sigma}}\left(\bigcap_{i=1}^mI_i^{(\lceil sa_i\rceil)}\right).\]
As well as Theorem \ref{thm:icsumsofprodsym} implies that,
for a subset $\Sigma\subseteq \H_n$, the integral closure of $E(\Sigma)^s=E(\Sigma^s)$ is equal to
\[\sum_{\ab=(a_1,\ldots ,a_n)\in P_{\Sigma}}\left(\bigcap_{i=1}^nJ_i^{(\lceil sa_i\rceil)}\right).\]
As well as Theorem \ref{thm:icsumsofprodpf} implies that,
for a subset $\Sigma\subseteq \H_{\lfloor n/2\rfloor}$, the integral closure of $F(\Sigma)^s=F(\Sigma^s)$ is equal to
\[\sum_{\ab=(a_1,\ldots ,a_{\lfloor n/2\rfloor})\in P_{\Sigma}}\left(\bigcap_{i=1}^{\lfloor n/2\rfloor}P_{2i}^{(\lceil sa_i\rceil)}\right).\]
So, exploiting Examples \ref{ex:generic+}, \ref{ex:symmetric+} and \ref{ex:skew-symmetric+}, the ideals $D(\Sigma)$, $E(\Sigma)$ and $F(\Sigma)$, introduced in \ref{sec:generic}, \ref{sec:symmetric} and \ref{sec:skew-symmetric} all satisfy condition ($\ast$).
\end{example}

\begin{thm}\label{thm:main2}
If $K$ has positive characteristic and $I\subseteq S$ is an ideal enjoying the condition ($\ast$) as in Definition \ref{def:ast}, then
\[\tau(\lambda\b I)=\sum_{\ab=(a_1,\ldots ,a_k)\in P}\left(\bigcap_{i=1}^k\pp_i^{(\lfloor \lambda a_i\rfloor +1-\height(\pp_i))}\right) \ \ \ \forall \ \lambda\in \RR_{>0}.\]
\end{thm}
\begin{proof}
Let us fix $\lambda\in\RR_{>0}$. 
First let us focus on the inclusion ``$\subseteq$". For any $i\in\{1,\ldots ,k\}$ and $\ab=(a_1,\ldots ,a_k)$, since $I$ satisfies condition ($\ast$), it is enough to show that 
\[\pp_i^{\big(\big\lceil \lceil \lambda q\rceil a_i\big\rceil\big)} \subseteq \bigg(\pp_i^{(\lfloor \lambda a_i\rfloor +1-\height(\pp_i))}\bigg)^{[q]} \ \ \ \mbox{for } \ q\gg 0,\]
where $q$ is a power of $\chara(K)=p$. To see this, let us take $q=p^e$ and localize at $\pp_i$.
Because $S_{\pp_i}$ is a regular local ring of dimension $\height(\pp_i)$, by using Lemma \ref{lem:maxreg} we infer that
\begin{eqnarray*}
(\pp_i S_{\pp_i})^{\big(\big\lceil \lceil \lambda q\rceil a_i\big\rceil\big)} & \subseteq & (\pp_i S_{\pp_i})^{\lceil \lambda a_iq\rceil} \\ 
& \underset{q\gg 0}{\subseteq} & \bigg((\pp_i S_{\pp_i})^{\lfloor \lambda a_i\rfloor +1-\height(\pp_i)}\bigg)^{[q]}
\end{eqnarray*}
So, when $q\gg 0$ we obtain that:
\[(\pp_i S_{\pp_i})^{\big(\big\lceil \lceil \lambda q\rceil a_i\big\rceil\big)} \subseteq \bigg((\pp_i S_{\pp_i})^{\lfloor \lambda a_i\rfloor +1-\height(\pp_i)}\bigg)^{[q]}.\]
By the flatness of the Frobenius over $S$, by intersecting back with $S$ we get:
\[\pp_i^{\big(\big\lceil \lceil \lambda q\rceil a_i\big\rceil\big)} \subseteq \bigg(\pp_i^{(\lfloor \lambda a_i\rfloor +1-\height(\pp_i))}\bigg)^{[q]},\]
which is what we wanted.

Let us now focus on the other inclusion. For a vector $\ab=(a_1,\ldots ,a_k)\in P$, take 
\[f\in \bigcap_{i=1}^k\pp_i^{(\lfloor \lambda a_i\rfloor +1-\height(\pp))}.\] 
Consider $F$ and $\prec$ as in the definition of the condition ($\ast$). For any $i=1,\ldots ,k$ and $q=p^e$, notice that:
\begin{eqnarray*}
F^{q-1}\cdot f^q & \in & \bigg(\pp_i^{(\height(\pp_i))}\bigg)^{q-1} \cdot \bigg(\pp_i^{(\lfloor \lambda a_i\rfloor +1-\height(\pp_i))}\bigg)^{q} \\
& \subseteq & \pp_i^{((q-1)\height(\pp_i)+q(\lfloor \lambda a_i\rfloor +1-\height(\pp_i)))} \\
& = & \pp_i^{(q\lfloor \lambda a_i\rfloor +q-\height(\pp_i))} \\
& = & \pp_i^{\bigg(q\bigg(\lfloor \lambda a_i\rfloor +\frac{q-\height{\pp_i}}{q}\bigg)\bigg)}.
\end{eqnarray*}
If $q$ is big enough, then 
\[q\bigg(\lfloor \lambda a_i\rfloor +\frac{q-\height{\pp_i}}{q}\bigg)\geq \lceil q\lambda a_i\rceil.\]
By \cite[Proposition 2.14]{BMS:MMJ}, we can assume that $q\lambda$ is an integer, and so we will do from now on.
So, let us fix $q$ big enough so that 
\[F^{q-1}f^q\in \pp_i^{(\lceil q\lambda a_i\rceil)} \ \ \ \forall \ i\in\{1,\ldots ,k\}.\] 
Therefore, because $I$ satisfies condition ($\ast$)
\[F^{q-1}f^q\in \overline{I^{q\lambda}}.\]
Take a positive integer $k$ such that 
\[\bigg(\overline{I^{q\lambda}}\bigg)^{k+\ell}\subseteq I^{q\ell\lambda}  \ \ \ \forall \ \ell\in \NN.\]
In particular, if $q'$ a power of $p$, we have
\[F^{(q-1)(q'+k)}f^{q(q'+k)}\in I^{qq'\lambda}.\]
Let $\B_{qq'}$ be the basis of $S$ over $S^{qq'}$ consisting in monomials. Remembering that $q$ has been fixed, we can choose $q'$ big enough such that 
\[v:=\init_{\prec}(F^{(q-1)(q'+k)}f^{qk})=\init_{\prec}(F)^{(q-1)(q'+k)}\init_{\prec}(f)^{qk}\in\B_{qq'}.\]
In fact, it is enough to take $q'>qk(\deg(f)+1)$. So
\[F^{(q-1)(q'+k)}f^{qk}=v+\sum_{\substack{u\in\B_{qq'} \\ u\prec v}}g_u^{qq'}u.\]
Therefore, we get
\[F^{(q-1)(q'+k)}f^{q(q'+k)}=f^{qq'}v+\sum_{\substack{u\in\B_{qq'} \\ u\prec v}}(f g_u)^{qq'}u,\]
from which we deduce that $f\in (I^{\lceil qq'\lambda \rceil})^{[1/qq']}$ by using \cite[Proposition 2.5]{BMS:MMJ}. So
\[f\in \tau(\lambda\b I).\]
\end{proof}

Theorem \ref{thm:main2}, together with Example \ref{ex:detast}, has the following corollaries:

\begin{corollary}\label{cor:genericsum}
With the notation of \ref{sec:generic}, for all $\Sigma\subseteq \H_m$ we have
\[\tau\big(\lambda\b D(\Sigma)\big)=\sum_{\ab=(a_1,\ldots ,a_m)\in P_{\Sigma}}\left(\bigcap_{i=1}^mI_i^{(\lfloor \lambda a_i\rfloor +1-(m-i+1)(n-i+1))}\right) \ \ \ \forall \ \lambda\in \RR_{>0}.\]
Equivalently, by using Theorem \ref{thm:symb}, $\tau\big(\lambda\b D(\Sigma)\big)$ is generated by the products of minors of $X$ whose shape $\alpha$ satisfies:
\[\gamma_i(\alpha)\geq \lfloor\lambda a_i\rfloor +1-(m-i+1)(n-i+1) \ \ \mbox{for some } \ab=(a_1,\ldots ,a_m)\in P_{\Sigma} \mbox{ and } \forall \ i=1,\ldots ,m.\]
In particular (independently on the characteristic!):
\[\fpt(D(\Sigma))=\max_{\ab=(a_1,\ldots ,a_m)\in P_{\Sigma}}\bigg\{\min\bigg\{\frac{(m-i+1)(n-i+1)}{a_i}:i=1,\ldots ,m\bigg\}\bigg\}.\]
\end{corollary}

\begin{corollary}\label{cor:symmetricsum}
With the notation of \ref{sec:symmetric}, $\Sigma\subseteq \H_n$ we have
\[\tau\big(\lambda\b E(\Sigma)\big)=\sum_{\ab=(a_1,\ldots ,a_n)\in P_{\Sigma}}\left(\bigcap_{i=1}^nJ_i^{(\lfloor \lambda a_i\rfloor +1-\binom{n-i+2}{2})}\right) \ \ \ \forall \ \lambda\in \RR_{>0}.\]
Equivalently, by using Theorem \ref{thm:symbsym}, $\tau\big(\lambda\b E(\Sigma)\big)$ is generated by the products of minors of $Y$ whose shape $\alpha$ satisfies:
\[\gamma_i(\alpha)\geq \lfloor\lambda a_i\rfloor +1-\binom{n-i+2}{2} \ \ \mbox{for some } \ab=(a_1,\ldots ,a_n)\in P_{\Sigma} \mbox{ and } \forall \ i=1,\ldots ,n.\]
In particular (independently on the characteristic!):
\[\fpt(E(\Sigma))=\max_{\ab=(a_1,\ldots ,a_n)\in P_{\Sigma}}\bigg\{\min\bigg\{\frac{\binom{n-i+2}{2}}{a_i}:i=1,\ldots ,n\bigg\}\bigg\}.\]
\end{corollary}

\begin{corollary}\label{cor:skew-symmetricsum}
With the notation of \ref{sec:skew-symmetric}, $\Sigma\subseteq \H_{\lfloor n/2\rfloor}$ we have
\[\tau\big(\lambda\b F(\Sigma)\big)=\sum_{\ab=(a_1,\ldots ,a_{\lfloor n/2\rfloor})\in P_{\Sigma}}\left(\bigcap_{i=1}^{\lfloor n/2\rfloor}P_{2i}^{(\lfloor \lambda a_i\rfloor +1-(n/2-i+1)(n-2i+1))}\right) \ \ \ \forall \ \lambda\in \RR_{>0}.\]
Equivalently, by using Theorem \ref{thm:symbpf}, $\tau\big(\lambda\b F(\Sigma)\big)$ is generated by the products of Pfaffians of $Z$ whose shape $\alpha$ satisfies:
\[\gamma_i(\alpha)\geq \lfloor\lambda a_i\rfloor +1-(n/2-i+1)(n-2i+1) \ \ \mbox{for some } \ab=(a_1,\ldots ,a_{\lfloor n/2\rfloor})\in P_{\Sigma} \mbox{ and } \forall \ i=1,\ldots ,\lfloor n/2\rfloor.\]
In particular (independently on the characteristic!):
\[\fpt(F(\Sigma))=\max_{\ab=(a_1,\ldots ,a_{\lfloor n/2\rfloor})\in P_{\Sigma}}\bigg\{\min\bigg\{\frac{(n/2-i+1)(n-2i+1)}{a_i}:i=1,\ldots ,\lfloor n/2\rfloor\bigg\}\bigg\}.\]
\end{corollary}

Now, we are ready to state the explicit formulas for the multiplier ideals of any $G$-invariant ideal in $\Sym(V\otimes W)$, in $\Sym(\Sym^2V)$ and in $\Sym(\bigwedge^2V)$, (recalled in Sections \ref{sec:generic}, \ref{sec:symmetric} and \ref{sec:skew-symmetric}). 

\begin{thm}\label{thm:mult}
Let $K$ be a field of characteristic $0$, $\Sigma\subseteq \P_m$ and $P\subseteq \RR^m$ be the convex hull of the set $\{(\gamma_1(\tl\sigma), \ldots ,\gamma_m(\tl\sigma)):\sigma\in\Sigma\}$. Then the ideal $I(\Sigma)\subseteq \Sym(V\otimes W)$ has multiplier ideals given by:
\[\Jc\big(\lambda\b I(\Sigma)\big)=\sum_{\ab=(a_1,\ldots ,a_m)\in P}\left(\bigcap_{i=1}^mI_i^{(\lfloor \lambda a_i\rfloor +1-(m-i+1)(n-i+1))}\right) \ \ \ \forall \ \lambda\in \RR_{>0}.\]
Equivalently, by using Theorem \ref{thm:symb}, $\Jc\big(\lambda\b I(\Sigma)\big)$ is generated by the products of minors of $X$ whose shape $\alpha$ satisfies:
\[\gamma_i(\alpha)\geq \lfloor\lambda a_i\rfloor +1-(m-i+1)(n-i+1) \ \ \mbox{for some } \ab=(a_1,\ldots ,a_m)\in P_{\Sigma} \mbox{ and } \forall \ i=1,\ldots ,m.\]
In particular:
\[\lct(I(\Sigma))=\max_{\ab=(a_1,\ldots ,a_m)\in P_{\Sigma}}\bigg\{\min\bigg\{\frac{(m-i+1)(n-i+1)}{a_i}:i=1,\ldots ,m\bigg\}\bigg\}.\]
\end{thm}
\begin{proof}
By Theorem \ref{thm:relinvdet} we have 
\[\overline{I(\Sigma)}=\overline{D(\tl\Sigma)},\]
where $\tl\Sigma=\{\tl\sigma:\sigma\in\Sigma\}$. So we have:
\[\Jc\big(\lambda\b I(\Sigma)\big)=\Jc\big(\lambda\b D(\tl\Sigma)\big)\]
(cf. \cite[Corollary 9.6.17]{La}).
However, we defined the ideal $D(\tl\Sigma)$ also in positive characteristic $p$ (where it is the reduction mod $p$ of $D(\tl\Sigma)$ viewed in characteristic 0). If $_p$ denotes the reduction mod $p$, by \eqref{eq:HY} we therefore obtain: 
\begin{equation}\label{eq:multest}
\Jc\big(\lambda\b D(\tl\Sigma)\big)_p=\tau\big(\lambda\b D(\tl\Sigma)_p\big)
\end{equation}
for $p\gg 0$ (a priori depending on $\lambda$). Therefore, by Corollary \ref{cor:genericsum}, a product of minors of shape $\sigma$ in $S$ belongs to $\Jc\big(\lambda\b D(\tl\Sigma)\big)_p$ if and only if there exists $a=(a_1,\ldots ,a_m)\in P_{\tl\Sigma}$ such that $\gamma_i(\sigma)\geq \lambda a_i+1-(m-i+1)(n-i+1)$ for all $i=1,\ldots ,m$ (independently on $p$). This implies that the multiplier ideal $\Jc\big(\lambda\b I(\Sigma)\big)=\Jc\big(\lambda\b D(\tl\Sigma)\big)$ is generated by the product of minors above, and because $P=P_{\tl\Sigma}$ the thesis follow.
\end{proof}

The same proof as above yields the analog result for $\Sym(\Sym^2V)$ and in $\Sym(\bigwedge^2V)$:

\begin{thm}\label{thm:multsym}
Let $K$ be a field of characteristic $0$, $\Sigma\subseteq \P_m$ and $P'\subseteq \RR^n$ be the convex hull of the set $\{(\gamma_1(\tl\sigma'), \ldots ,\gamma_n(\tl\sigma')):\sigma\in\Sigma\}$, where $\sigma'_i=\tl\sigma_{2i}$. Then the ideal $J(\Sigma)\subseteq \Sym(\Sym^2V)$ has multiplier ideals given by:
\[\Jc\big(\lambda\b J(\Sigma)\big)=\sum_{\ab=(a_1,\ldots ,a_n)\in P'}\left(\bigcap_{i=1}^nJ_i^{(\lfloor \lambda a_i\rfloor +1-\binom{n-i+2}{2})}\right) \ \ \ \forall \ \lambda\in \RR_{>0}.\]
Equivalently, by using Theorem \ref{thm:symbsym}, $\Jc\big(\lambda\b J(\Sigma)\big)$ is generated by the products of minors of $Y$ whose shape $\alpha$ satisfies:
\[\gamma_i(\alpha)\geq \lfloor\lambda a_i\rfloor +1-\binom{n-i+2}{2} \ \ \mbox{for some } \ab=(a_1,\ldots ,a_n)\in P_{\Sigma} \mbox{ and } \forall \ i=1,\ldots ,n.\]
In particular:
\[\lct(J(\Sigma))=\max_{\ab=(a_1,\ldots ,a_n)\in P_{\Sigma}}\bigg\{\min\bigg\{\frac{\binom{n-i+2}{2}}{a_i}:i=1,\ldots ,n\bigg\}\bigg\}.\]
\end{thm}

\begin{thm}\label{thm:multpf}
Let $K$ be a field of characteristic $0$, $\Sigma\subseteq \P_{\lfloor n/2\rfloor}\cap \Cc_e$ and $\widetilde{P}\subseteq \RR^n$ be the convex hull of the set $\{(\gamma_1(\widetilde{\sigma}), \ldots ,\gamma_m(\widetilde{\sigma})):\sigma\in\Sigma\}$, where $\widetilde{\sigma}_i=\tl\sigma_{i}/2$. Then the ideal $P(\Sigma)\subseteq \Sym(\bigwedge^2V)$ has multiplier ideals given by
\[\Jc\big(\lambda\b P(\Sigma)\big)=\sum_{\ab=(a_1,\ldots ,a_{\lfloor n/2\rfloor})\in \widetilde{P}}\left(\bigcap_{i=1}^{\lfloor n/2\rfloor}P_{2i}^{(\lfloor \lambda a_i\rfloor +1-(n/2-i+1)(n-2i+1))}\right) \ \ \ \forall \ \lambda\in \RR_{>0}.\]
Equivalently, by using Theorem \ref{thm:symbpf}, $\Jc\big(\lambda\b P(\Sigma)\big)$ is generated by the products of Pfaffians of $Z$ whose shape $\alpha$ satisfies:
\[\gamma_i(\alpha)\geq \lfloor\lambda a_i\rfloor +1-(n/2-i+1)(n-2i+1) \ \ \mbox{for some } \ab=(a_1,\ldots ,a_{\lfloor n/2\rfloor})\in P_{\Sigma} \mbox{ and } \forall \ i=1,\ldots ,\lfloor n/2\rfloor.\]
In particular:
\[\lct(P(\Sigma))=\max_{\ab=(a_1,\ldots ,a_{\lfloor n/2\rfloor})\in P_{\Sigma}}\bigg\{\min\bigg\{\frac{(n/2-i+1)(n-2i+1)}{a_i}:i=1,\ldots ,\lfloor n/2\rfloor\bigg\}\bigg\}.\]
\end{thm}

\begin{remark}\label{rem:mon}
To conclude, another class of ideals of $S=K[x_1,\ldots ,x_N]$ satisfying the condition ($\ast$) of Definition \ref{def:ast} is the class of monomial ideals $I$. With the notation of Definition \ref{def:ast}, $\pp_1=(x_1)$, $\ldots$, $\pp_N=(x_N)$ and $P\subseteq \RR^N$ is the Newton polytope $\mathrm{NP}(I)$ of $I$, that is the convex hull of the exponents corresponding to a minimal system of monomial generators of $I$ (cf. \cite[Proposition 3.4]{Tessier}. The polynomial $F\in S$ doing the job is just $F=x_1\cdots x_N$, and any term ordering is good.

By Theorem \ref{thm:main2}, so, if $K$ has positive characteristic and $I\subseteq S$ is a monomial ideal, $\forall \ \lambda\in \RR_{>0}$ we have:
\[\tau(\lambda\b I)=(x_1^{\lfloor \lambda a_1\rfloor}\cdots x_N^{\lfloor \lambda a_N\rfloor}:(a_1,\ldots ,a_N)\in \mathrm{NP}(I))=(x_1^{b_1}\cdots x_N^{b_N}:(b_1+1,\ldots ,b_N+1)\in \lambda\cdot\mathrm{NP}(I)\cap\ZZ^N).\]
Notice also that ideals defined by a single monomial have floating test ideals.

In characteristic 0, by exploiting \eqref{eq:HY} as in the proof of Theorem \ref{thm:mult}, we recover the formula of Howald \cite{Howald} for the multiplier ideals of monomial ideals (see also \cite[Section 9.3.C]{La}:
\[\Jc(\lambda\b I)=(x_1^{b_1}\cdots x_N^{b_N}:(b_1+1,\ldots ,b_N+1)\in \lambda\cdot\mathrm{NP}(I)\cap\ZZ^N).\]
\end{remark}

\end{document}